\documentclass[11pt]{amsart}
\usepackage{amssymb}
\usepackage{amsmath,url}

\title{Semiregular automorphisms of edge-transitive graphs}

\author[Michael Giudici]{Michael Giudici}
\address{Michael Giudici, Centre for Mathematics of Symmetry and Computation, \newline
\indent School of Mathematics and Statistics, University of Western Australia, \newline
\indent 35 Stirling Highway, Crawley, WA 6009, Australia.}
\email[Michael Giudici]{giudici@maths.uwa.edu.au}

\author[Primo\v{z} Poto\v{c}nik]{Primo\v{z} Poto\v{c}nik}
\address{Primo\v{z} Poto\v{c}nik, Faculty of Mathematics and Physics, University of Ljubljana, \newline
\indent Jadranska 21, SI-1000 Ljubljana, Slovenia.\newline
\indent Also affiliated with: IAM, University of Primorska, \newline
\indent Glagolja\v{s}ka 8, SI-6000 Koper, Slovenia.}
\email {primoz.potocnik@fmf.uni-lj.si}
             
\author[Gabriel Verret]{Gabriel Verret}
\address{Gabriel Verret, Centre for Mathematics of Symmetry and Computation, \newline
\indent School of Mathematics and Statistics, University of Western Australia, \newline
\indent 35 Stirling Highway, Crawley, WA 6009, Australia. \newline
\indent Also affiliated with: FAMNIT, University of Primorska, \newline
\indent Glagolja\v{s}ka 8, SI-6000 Koper, Slovenia.}
\email{gabriel.verret@uwa.edu.au}

\newtheorem{theorem}{Theorem}[section]

\newtheorem{lemma}[theorem]{Lemma}
\newtheorem{corollary}[theorem]{Corollary}

\newtheorem{question}[theorem]{Question}

\newcommand{\K}{{\rm{K}}}
\newcommand{\W}{{\rm{W}}}
\newcommand{\V}{{\rm{V}}}
\newcommand{\E}{{\rm{E}}}
\newcommand{\ZZ}{\mathbb{Z}}
\newcommand{\DD}{\mathbb{D}}

\newcommand{\Aut}{{\rm{Aut}}}
\newcommand{\Sym}{{\rm Sym}}

\thanks{This research began during a visit of the first author to Ljubljana. He thanks the second and third authors for their hospitality and support of the visit. The research of the first author is supported by an ARC Discovery project. The third author is supported by UWA as part of the Australian Research Council grant DE130101001. }


\begin{document}

\begin{abstract}
The polycirculant conjecture asserts that every vertex-transitive digraph has a semiregular automorphism, that is, a nontrivial automorphism whose cycles all have the same length. In this paper we investigate the existence of semiregular automorphisms of edge-transitive graphs. In particular, we show that any regular edge-transitive graph of valency three or four has a semiregular automorphism. 
\end{abstract}

\maketitle

\section{Introduction}

All graphs considered in this paper are finite and simple. A graph $\Gamma$ is said to be $G$-\emph{vertex-transitive} if $G$ is a subgroup of $\Aut(\Gamma)$ acting transitively on $\V(\Gamma)$. Similarly, $\Gamma$ is said to be $G$-\emph{arc-transitive} or $G$-\emph{edge-transitive} if $G$ acts transitively on the arcs or edges of $\Gamma$, respectively. (An \emph{arc} is an ordered pair of adjacent vertices.) Furthermore, a graph is said to be {\em locally $G$-arc-transitive} provided that, for every vertex $v$ of $\Gamma$, the permutation group $G_v^{\Gamma(v)}$ induced by the action of the vertex-stabiliser $G_v$  on the neighbourhood $\Gamma(v)$ is transitive. When $G=\Aut(\Gamma)$, the prefix $G$ in the above notation is sometimes omitted.

For a group $G$ that acts on a set $\Omega$ we say that it is \emph{semiregular} on $\Omega$ if the stabiliser $G_\omega$ is trivial for every $\omega\in \Omega$; note that this implies that the action is faithful. A non-identity element $g\in G$ is called {\em semiregular} (on $\Omega$), provided that the group $\langle g \rangle$ generated by $g$ is semiregular on $\Omega$.  Equivalently, a non-identity element $g\in G$ is semiregular provided that $\langle g \rangle$ acts faithfully on $\Omega$ and that all of the cycles of the permutation induced by $g$ have the same length.

Existence of a graph automorphism that acts semiregularly on the vertex-set of the graph (we shall refer to such an automorphism simply as a {\em semiregular automorphism}) is often a desirable feature of a graph, from both a theoretical and a practical point of view. For example they give nice representations of the graphs \cite{biggs}, are used in the enumeration of graphs of small order \cite{MR} and in the construction of Hamiltonian paths and cycles \cite{alspach}.

In 1981, Maru\v{s}i\v{c} conjectured~\cite{Dragan} that every vertex-transitive digraph with at least $2$ vertices has a semiregular automorphism. Even though the conjecture has been settled for some special classes of vertex-transitive graphs (for example, for vertex-transitive graphs of prime-power order \cite{Dragan} and order twice the square of a prime \cite{DraganScap}, for graphs of valency at most 4 \cite{Dobson,DraganScap}, for distance transitive graphs~\cite{KS}, for locally quasiprimitive arc-transitive graphs~\cite{GX}, and for some other special families of arc-transitive graphs~\cite{GV,Verret,X}), the complete solution still seems to be beyond our reach.

Recently, the interest in semiregular automorphisms has spread to other areas of combinatorics. For example, Pisanski raised the question~\cite{conj} whether every flag-transitive configuration of type $(v_r)$ is polycyclic (see \cite{BP} for the definition of and basic facts about polycyclic configurations).

This motivates the following graph-theoretical question, whose affirmative answer would imply an affirmative answer to the question of Pisanski (see \cite{MarPis} for an overview of the connection between symmetry properties of combinatorial configurations and graphs).

\begin{question}\label{q:GPV}
Does every connected, regular, locally arc-transitive graph admit a semiregular automorphism? Moreover, if the graph is bipartite, can the semiregular automorphism be chosen so as to fix each part of the bipartition?
\end{question}

As we prove in this paper, the answer to this question is positive for regular graphs of valency at most $4$. In fact, we prove slightly more general statements, in the setting of edge-transitive graphs. Note that regular is required as the complete bipartite graph $K_{3,4}$ does not admit a semiregular automorphism.

\begin{theorem}\label{theorem:main3}
Let $\Gamma$ be a connected $G$-edge-transitive $3$-valent graph. Then $G$ contains a semiregular element.
\end{theorem}

Formulation of the analogous result for $4$-valent graphs requires some definitions. The \emph{wreath graph} $\W(n,2)$ is the lexicographic product of a cycle of length $n\geq 3$ and an edgeless graph on $2$ vertices. In other words, $\V(\W(n,2))=\ZZ_n\times\ZZ_2$ with $(i,u)$ being adjacent to $(j,v)$ if and only if $i-j\in\{-1,1\}$. The {\em subdivided double} $\DD\Lambda$ of a graph $\Lambda$ is the bipartite graph with parts $\E(\Lambda)$ and $\V(\Lambda)\times\ZZ_2$. An edge $e\in \E(\Lambda)$ is adjacent in $\DD\Lambda$ to a pair $(v,i)\in \V(\Lambda)\times\ZZ_2$ whenever $v$ is an endvertex of $e$ in $\Lambda$. 

\begin{theorem}\label{theorem:main}
Let $\Gamma$ be a connected $G$-edge-transitive $4$-valent graph. Then at least one of the following holds:
\begin{enumerate}
\item $G$ contains a semiregular element, 
\item $\Gamma\cong \W(2^n,2)$ for some $n\geq 2$, \label{ex1}
\item $\Gamma\cong \DD\Lambda$ for some arc-transitive, $4$-valent graph $\Lambda$ of order a power of $2$. \label{ex2}
\end{enumerate}
\end{theorem}

Using Theorems~\ref{theorem:main3} and~\ref{theorem:main}, it is not difficult to prove the following result, which answers Question~\ref{q:GPV} for graphs of valency at most $4$.

\begin{corollary}\label{cor:poly}
Let $\Gamma$ be a connected edge-transitive regular graph of valency at most $4$. Then $\Gamma$ admits a semiregular automorphism. Moreover, if $\Gamma$ is bipartite and locally arc-transitive, the semiregular element can be chosen in such a way as to preserve each part of the bipartition.
\end{corollary}

Note that in the $3$-valent case, we can guarantee that any edge-transitive group contains a semiregular element, while in the $4$-valent case, we only claim this for the full automorphism group. In fact, the exceptional cases~(\ref{ex1}) and~(\ref{ex2}) in Theorem~\ref{theorem:main} cannot be avoided, as will be shown in  Section~\ref{sec:counterexample}.

\section{Preliminaries}
\label{sec:prelim}

In this section we prove a few auxiliary results. We start with the following lemma, the proof of which is easy and left to the reader.

\begin{lemma}\label{lem:final}
A connected locally $G$-arc-transitive graph is $G$-edge-transitive.
\end{lemma}
%\begin{proof}
%Let $\Gamma$ be a connected locally $G$-arc-transitive graph. Let $\{u,v\}$ and $\{v,w\}$ be two adjacent edges of $\Gamma$. Since $G_v^{\Gamma(v)}$ is transitive, $\{u,v\}$ and $\{v,w\}$ are in the same $G$-orbit. Using induction and connectedness of $\Gamma$, it follows that $\Gamma$ is $G$-edge-transitive.
%\end{proof}

Next, we have the following well-known classification of $G$-edge-transitive graphs:

\begin{lemma}
\label{lem:Get}
Let $\Gamma$ be a connected $G$-edge-transitive graph. Then one of the following occurs:
\begin{enumerate}
\item $\Gamma$ is $G$-arc-transitive,
\item $\Gamma$ is $G$-vertex-transitive and $G_v^{\Gamma(v)}$ has two orbits of equal size for every vertex $v$,
\item $\Gamma$ is bipartite and the two sets of the bipartition are exactly the orbits of $G$ on $\V(\Gamma)$.
\end{enumerate}
\end{lemma}

Note that, if (1) or (3) of Lemma~\ref{lem:Get} occurs, then $\Gamma$ is locally $G$-arc-transitive while if (2) occurs, then $\Gamma$ is $G$-half-arc-transitive. Lemma~\ref{lem:Get} also has the following two immediate corollaries.

\begin{corollary}
\label{cor:Get2}
Let $\Gamma$ be a connected $G$-edge-transitive graph. Then, if there is a path of even length between two vertices, they are in the same $G$-orbit. 
\end{corollary}

\begin{corollary}
\label{cor:Get3}
Let $\Gamma$ be a connected $G$-edge-transitive graph. Then $\Gamma$ is either locally $G$-arc-transitive, or it is $G$-vertex-transitive and $G_v^{\Gamma(v)}$ has two orbits of equal size (and, in particular, the valency of $\Gamma$ is even).
\end{corollary}

We also need the following folklore lemma:

\begin{lemma} \label{lem:leash} 
Let $p$ be a prime, let $\Gamma$ be a connected graph and let $G\le \Aut(\Gamma)$. If, for every $v\in \V(\Gamma)$, $|G_v^{\Gamma(v)}|$ is coprime to $p$, then so is $|G_u|$ for every $u\in \V(\Gamma)$. 
\end{lemma}
\begin{proof}
We prove this by contradiction. Let $p$ be a prime dividing $|G_u|$ for some $u\in\V(\Gamma)$ and let $g\in G_u$ have order $p$. Since $g$ is non-trivial, it must move some vertex. Let $w$ be a vertex of $\Gamma$ moved by $g$ at minimal distance from $u$. By the connectivity of $\Gamma$, there is a path $u,u_1,\ldots,u_t,w$ such that $g$ fixes each $u_i$. Then $g\in G_{u_t}$ and $g$ acts nontrivially on $\Gamma(u_t)$. Thus $p$ divides $|G_{u_t}^{\Gamma(u_t)}|$ and the result follows.
\end{proof}

\begin{lemma}
\label{lem:Nu}
Let $\Gamma$ be a connected $G$-edge-transitive graph.  Let $\{u,v\}$ be an edge of $\Gamma$ and let $N$ be a normal subgroup of $G$, such that $N_v^{\Gamma(v)}$ is transitive. Let $A$ be the set of vertices $a$ of $\Gamma$ such that there is a path of even length between $u$ and $a$. Then $N$ is transitive on $A$.
\end{lemma}
\begin{proof}
Let $a\in A$. Then there is a path of even length $2\ell$ between $u$ and $a$. We show that $a\in u^N$ by induction on $\ell$. This is clearly true if $\ell=0$. Assume that the statement is true for some $\ell\geq 0$. Let $y\in u^G$ such that there is a path $P$ of length $2\ell+2$ between $u$ and $y$. Let $x$ be the vertex before $y$ on $P$, and let $w$ be the vertex before $x$. There is a path of length $2l$ between $u$ and $w$ and hence, by the induction hypothesis, $w\in u^N$. There is also a path of even length between $v$ and $x$ hence, by Corollary~\ref{cor:Get2}, $x\in v^G$ and hence $N_x^{\Gamma(x)}$ is transitive. It follows that there is an element of $N$ mapping $w$ to $y$ and therefore $y\in u^N$, completing the induction. 
\end{proof}

\begin{lemma}
\label{cor:powerorder}
Let $p$ be a prime and let $\Gamma$ be a connected $G$-edge-transitive regular graph. Let $u$ be a vertex of $\Gamma$ and let $A$ be the set of vertices $a$ of $\Gamma$ such that there is a path of even length between $u$ and $a$. Suppose that $|A|$ is a power of $p$. Then either $G$ contains a semiregular element, or $\Gamma$ is bipartite and there exist $p$ vertices of $\Gamma$ having the same neighbourhood.
\end{lemma}
\begin{proof}
Let $B$ be the set of vertices $b$ of $\Gamma$ such that there is a path of odd length between $u$ and $b$. If $\Gamma$ is not bipartite, then $A=B=\V(\Gamma)$, while if $\Gamma$ is bipartite, the parts are $A$ and $B$ and, since $\Gamma$ is regular, $|A|=|B|$. 

 Let $S$ be a Sylow $p$-subgroup of $G$. By \cite[Theorem 3.4']{wielandt}, $S$ is transitive on both $A$ and $B$. Let $z$ be an element of order $p$ in the centre of $S$. Then, for each $S$-orbit, $z$ acts either trivially or semiregularly on it. Since $z$ has order $p$, it cannot act trivially on both $A$ and $B$. If $z$ acts semiregularly on both $A$ and $B$, then $z$ is a semiregular automorphism contained in $G$ and the result holds. Without loss of generality, we may thus assume that $z$ acts trivially on $A$ and semiregularly on $B$ and hence $A\neq B$ and $\Gamma$ is bipartite.

Let $b\in B$. All the neighbours of $b$ are contained in $A$, and thus $z$ fixes all of them. It follows that all the vertices in the $\langle z\rangle$-orbit of $b$ have the same neighbourhood. Since the order of $z$ is $p$, it follows that $|b^{\langle z\rangle}| = p$, which completes the proof of the lemma.
\end{proof}

 A graph is called {\em unworthy} if two of its vertices have the same neighbourhood. Unworthy $4$-valent edge-transitive graphs were classified in \cite{PW}. 

\begin{lemma}\cite[Lemma~4.3]{PW}
\label{lem:unworthy}
Let $\Gamma$ be a connected edge-transitive $4$-valent unworthy graph. Then one of the following occurs:
\begin{enumerate}
\item $\Gamma\cong \W(n,2)$ for some $n\ge 3$, 
\item $\Gamma\cong \DD\Lambda$ for some arc-transitive, $4$-valent graph $\Lambda$. 
\end{enumerate}
\end{lemma}

\begin{lemma}\cite[Lemma~2.5]{Verret}\label{lemma:main}
Let $p$ be an odd prime and let $\Gamma$ be a connected $G$-arc-transitive $4$-valent graph such that $p$ divides $|\V(\Gamma)|$. If $G$ is solvable, then $G$ contains a semiregular element of order $p$.
\end{lemma}

Let $G$ be a permutation group on a set $\Omega$ and let $K$ be a normal subgroup of $G$. Then the partition  $\Omega/K = \{\omega^K : \omega \in \Omega\}$ into the set of $K$-orbits is $G$-invariant. There is thus a natural induced action of $G$ on $\Omega/K$. Clearly, $K$ is always contained in the kernel of this action, and hence there is an induced action of $G/K$ on $\Omega/K$. The next lemma deals with the situation where the latter action is faithful; this is equivalent to saying that the kernel of the action of $G$ on $\Omega/K$ equals $K$.

\begin{lemma}\cite[Lemma~2.3]{Verret}\label{lemma:coprime}
Let $G$ be a permutation group acting on a set $\Omega$ and let $K$ be a normal subgroup of $G$ such that the induced action of $G/K$ on $\Omega/K$ is faithful. If $G/K$ contains an element of some order $r$ coprime to $|K|$ that acts semiregularly on $\Omega/K$, then $G$ contains an element of order $r$ acting semiregularly on $\Omega$.
\end{lemma}

Let $\Gamma$ be a graph and let $N\leq\Aut(\Gamma)$. The {\em quotient graph} $\Gamma/N$ is the graph whose vertices are the $N$-orbits with  two such $N$-orbits $v^N$ and $u^N$ adjacent whenever there is a pair of vertices $v'\in v^N$ and $u'\in u^N$ that are adjacent in $\Gamma$.  Clearly, if $\Gamma$ is connected then so is $\Gamma/N$. We will need a few facts about quotient graphs which we collect in the following lemma (see for example \cite[Lemmas 1.4, 1.5, 1.6 and Remark 1.8]{P85}).

\begin{lemma}\label{lem:quo}
Let $\Gamma$ be a connected locally $G$-arc-transitive graph and let $N$ be a normal subgroup of $G$. Let $K$ be the kernel of the action of $G$ on $N$-orbits. Then 
\begin{enumerate}
\item $G/K$ acts faithfully on $\Gamma/N$ and $\Gamma/N$ is locally $G/K$-arc-transitive, and \label{local}
\item the valency of a vertex $v^N$ in $\Gamma/N$ divides the number of orbits of $K_v^{\Gamma(v)}$. \label{val}
\item In particular, if $G_v^{\Gamma(v)}$ is $2$-transitive and the valency of $v^N$ is greater than $1$, then the valency of $v^N$ is the same as the valency of $v$ and $K_v^{\Gamma(v)}=1$. \label{val2}
\item If $G_v^{\Gamma(v)}$ is $2$-transitive, then $(G/K)_{v^N}^{(\Gamma/N)(v^N)}$ is also $2$-transitive. \label{2trans} 
\end{enumerate}
\end{lemma}

In the proof of the main results, we will need to consider the action of $\Aut(\Gamma)$ not only on the vertex-set, but also on the edge-set of $\Gamma$. Let us state a two easy results about this action.

\begin{lemma}\label{lemma:technical}
If $\Gamma$ is a connected graph with at least $3$ vertices, then $\Aut(\Gamma)$ acts faithfully on $\E(\Gamma)$.
\end{lemma}

The proof of Lemma~\ref{lemma:technical} is an easy exercise.

%\begin{proof}
%Suppose the contrary and choose an element
%$g$ of the kernel of the action of $\Aut(\Gamma)$ on $\E(\Gamma)$ which acts nontrivially on $\V(\Gamma)$. Let $v\in\V(\Gamma)$ be an arbitrary %vertex such that $v^g\neq v$
%and let $x$ be any of its neighbours.
%Since $g$ acts trivially on $\E(\Gamma)$, it follows that $\{v,x\}^g=\{v,x\}$; and since $v^g\not = v$, it follows that $x=v^g$.
%Since $x$ was an arbitrary neighbour of $v$, this implies that $x$ is also the only neighbour of $v$.

%This shows that $g$ fixes every vertex of valency at least $2$. Since $\Gamma$ is connected and has at least $3$ vertices,
%no two vertices of valency $1$ are adjacent, implying that every vertex of valency $1$ is adjacent to some vertex fixed by $g$. But then also %vertices of valency $1$
%are fixed by $g$, which contradicts our assumption that $g$ acts nontrivially on the vertices.
%\end{proof}

\begin{corollary}\label{cor:v+e}
Let $p$ be a prime and let $\Lambda$ be a connected graph with at least $3$ vertices such that $|\V(\Lambda)|$ and $|\E(\Lambda)|$ are both powers of $p$. Let $G\leqslant \Aut(\Lambda)$ such that $G$ is transitive on both $\V(\Lambda)$ and $\E(\Lambda)$. Then $G$ contains an element which acts semiregularly on each of $\V(\Lambda)$ and $\E(\Lambda)$.
\end{corollary}
\begin{proof}
Let $S$ be a Sylow $p$-subgroup of $G$. By \cite[Theorem 3.4']{wielandt}, $S$ acts transitively on both $\V(\Lambda)$ and $\E(\Lambda)$. Let $z$ be an element of order $p$ in the center of $S$. If $z$ fixes a vertex, then, being central in a vertex-transitive group, it must fix all vertices, which is a contradiction. Similarly, if $z$ fixes an edge, it must fix all edges and, by Lemma~\ref{lemma:technical}, this is a contradiction. We thus conclude that $z$ is semiregular on both $\V(\Lambda)$ and $\E(\Lambda)$.
\end{proof}

\section{Proofs of the main theorems}
\subsection{Proof of Theorem~\ref{theorem:main3}}
Let $\Gamma$ be a connected, $G$-edge-transitive, $3$-valent graph. We must show that $G$ contains a semiregular element. 

Since the valency of $\Gamma$ is odd, Corollary~\ref{cor:Get3} implies that $\Gamma$ is locally $G$-arc-transitive. Moreover, since $\Gamma$ is connected and is $3$-valent, it follows by Lemma~\ref{lem:leash} that $G_v$ is a  $\{2,3\}$-group for every vertex $v$. An element of $G$ of order a prime greater than $3$ is thus necessarily semiregular. We may thus assume that $G$ contains no such element and hence $G$ itself is a $\{2,3\}$-group. It follows by Burnside's Theorem that $G$ is solvable. In particular, $G$ contains a non-trivial normal elementary abelian $p$-group $N$ for some $p\in \{2,3\}$. If $N$ is semiregular, then any of its nontrivial elements is semiregular, and the conclusion of the theorem holds. We may thus assume that $N$ is not semiregular and hence $N_v^{\Gamma(v)}\neq 1$ for some vertex $v\in\V(\Gamma)$. Since  $N_v^{\Gamma(v)}$ is a normal subgroup of a transitive permutation group $G_v^{\Gamma(v)}$ of degree $3$, it follows that $N_v^{\Gamma(v)}$ is transitive and $p=3$. Let $u$ be a neighbour of $v$ and let $A$ be the set of vertices 
$a$ of $\Gamma$ such that there is a path of even length between $u$ and $a$. By Lemma~\ref{lem:Nu}, $N$ is transitive on $A$. In particular $|A|$ is a power of $3$. In view of Lemma~\ref{cor:powerorder}, it follows that $\Gamma\cong \K_{3,3}$. Since $G$ is edge-transitive, $9$ divides $|G|$ and hence $G$ contains the unique Sylow $3$-subgroup of $\Aut(\Gamma)$, which is easily seen to contain a semiregular element. This concludes the proof.

\subsection{Proof of Theorem~\ref{theorem:main}}

Let $\Gamma$ be a connected, $G$-edge-transitive, $4$-valent graph. Since $\Gamma$ is $4$-valent, it follows that for every $v\in \V(\Gamma)$, $G_v^{\Gamma(v)}$ is a subgroup of $\Sym(4)$ and thus a $\{2,3\}$-group. Lemma~\ref{lem:leash} then implies that $G_v$ is also a $\{2,3\}$-group for every vertex $v\in \V(\Gamma)$. An element of $G$ of order a prime greater than $3$ is thus necessarily semiregular. Hence we may assume that $G$ contains no such element. This establishes the following:
\medskip 

\noindent
{\bf Fact 1}:  $G$  is a $\{2,3\}$-group.
\medskip

If $\Gamma$ is bipartite, let $A$ and $B$ be the two parts. Otherwise, let $A=B=\V(\Gamma)$. By Corollary~\ref{cor:Get2}, $G$ is transitive on both $A$ and $B$ and $|A|=|B|$. Since $G$ is a $\{2,3\}$-group, no prime other than $2$ or $3$ can divide $|A|$. If $|A|$ is a power of $2$ then, by Lemma~\ref{cor:powerorder}, either $G$ contains a semiregular element, or $\Gamma$ is unworthy. In the latter case, Lemma~\ref{lem:unworthy} implies that $\Gamma$ is isomorphic to either a wreath graph, or to $\DD\Lambda$ for some arc-transitive, $4$-valent graph $\Lambda$. The result follows in each case. We may thus assume that $3$ divides $|A|$.

\medskip

\noindent
{\bf Fact 2}:  $3$ divides $|A|$.
\medskip

If $G_v$ is a $2$-group for every $v\in \V(\Gamma)$, then every element of order $3$ in $G$ is semiregular and the conclusion holds. We may thus assume that there exists a vertex $w$ such that $3$ divides $|G_w|$. By Lemma~\ref{lem:leash}, there exists a vertex $u$ such that $3$ divides $|G_u^{\Gamma(u)}|$. By Corollary~\ref{cor:Get3}, it follows that $\Gamma$ is locally $G$-arc-transitive. In particular, $4$ also divides $|G_u^{\Gamma(u)}|$, and hence $G_u^{\Gamma(u)}$ must be $2$-transitive. Hence:

\medskip 

\noindent
{\bf Fact 3}:  $\Gamma$ is locally $G$-arc-transitive and there exists a vertex $u$ such that $G_u^{\Gamma(u)}$ is $2$-transitive.

\medskip

Since $G$ is a $\{2,3\}$-group, Burnside's Theorem implies that $G$ is solvable and thus contains a non-trivial normal elementary abelian $p$-group $N$ for some $p\in\{2,3\}$. 

 Suppose that $N$ is transitive on $A$ or $B$. Then, since $3$ divides $|A|$, $p=3$ and $|A|$ is a power of $3$. In view of Lemma~\ref{cor:powerorder}, we may assume that $\Gamma$ is bipartite and at least three vertices of $\Gamma$ have the same neighbourhood.  It is not difficult to see that, since $\Gamma$ is edge-transitive, this implies that $\Gamma\cong K_{4,4}$, which contradicts the fact that $3$ divides $|\V(\Gamma)|$. This allows us to assume:
\medskip 

\noindent
{\bf Fact 4}: $N$ is intransitive on both $A$ and $B$.
\medskip

Together with Lemma~\ref{lem:Nu}, Fact 4 immediately implies that $N_v^{\Gamma(v)}$ is intransitive for every vertex $v\in\V(\Gamma)$. An intransitive normal subgroup of a $2$-transitive group is necessarily trivial, hence $N_u^{\Gamma(u)}=1$. Let $v$ be a neighbour of $u$. If $N_v^{\Gamma(v)}=1$ then, by Lemma~\ref{lem:leash}, $N$ is semiregular. We may thus assume that $N_v^{\Gamma(v)}\neq 1$. 

Thus $N_v^{\Gamma(v)}$ is a nontrivial intransitive normal elementary abelian $p$-subgroup of $G_v^{\Gamma(v)}$, which is itself a transitive subgroup of $\Sym(4)$. A simple examination of the transitive subgroups of $\Sym(4)$ yields that $p=2$, $G_v^{\Gamma(v)}$ is a $2$-group and $N_v^{\Gamma(v)}$ has two orbits. Moreover, since $G_v^{\Gamma(v)} \not \cong G_u^{\Gamma(u)}$, this also implies that $\Gamma$ is  bipartite. We thus have:

\medskip 

\noindent
{\bf Fact 5}:  $p=2$, $G_v^{\Gamma(v)}$ is a $2$-group, $N_u^{\Gamma(u)}=1$, $N_v^{\Gamma(v)}$ has two orbits of length $2$ and $\Gamma$ is bipartite with parts $A$ and $B$.
\medskip

Let us now consider the quotient graph $\Gamma/N$. Since $\Gamma$ is bipartite and $N$ preserves the parts, $\Gamma/N$ is bipartite with parts $A/N$ and $B/N$. Since $N$ is a $2$-group while $|A|=|B|$ is divisible by $3$, it follows that $|A/N|$ and $|B/N|$ are both divisible by $3$. Let $K$ be the kernel of the action of $G$ on the $N$-orbits. By Lemma~\ref{lem:quo}~(\ref{local}), $G/K$ acts faithfully on $\Gamma/N$ and $\Gamma/N$ is locally $G/K$-arc-transitive.

\medskip 

\noindent
{\bf Fact 6}:  $\Gamma/N$ is a connected locally $G/K$-arc-transitive bipartite graph with parts having size divisible by $3$. 
\medskip

In particular $\Gamma/N$ has minimal valency at least $2$. On the other hand, since $N_v^{\Gamma(v)}$ has two orbits, the valency of $v^N$ in $\Gamma/N$ is, by Lemma~\ref{lem:quo}~(\ref{val}), a divisor of $2$. In particular, the valency of $v^N$ is exactly $2$. 

Since  $G_u^{\Gamma(u)}$ is $2$-transitive, it follows from Lemma~\ref{lem:quo}~(\ref{val2}) that the valency of $u^N$ is $4$ and that $K_u^{\Gamma(u)}=1$. Since $G_v^{\Gamma(v)}$ is a $2$-group, $K_v^{\Gamma(v)}$ is also a $2$-group. By Lemma~\ref{lem:leash}, $K_v$ is a $2$-group and, since $K=NK_v$, $K$ is also a $2$-group.

\medskip 

\noindent
{\bf Fact 7}: $v^N$ has valency $2$, $u^N$ has valency $4$ and  $K$ is a $2$-group.
\medskip

 Suppose that two vertices in the $G/K$-orbit of $v^N$ have the same neighbourhood. Since $G_u^{\Gamma(u)}$  is $2$-transitive, it follows from Lemma~\ref{lem:quo}~(\ref{2trans}) that $(G/K)_{u^N}^{\Gamma/N(u^N)}$ is $2$-transitive and hence all neighbours of $u^N$ have the same neighbourhood. It follows that $\Gamma/N\cong K_{2,4}$, which contradicts the fact that $3$ divides the size of the parts of $\Gamma/N$. We may thus assume that vertices in the $G/K$-orbit of $v^N$ have distinct neighbours and hence $\Gamma/N$ is the subdivision of a connected simple $4$-valent graph $\Lambda$.

\medskip

\noindent
{\bf Fact 8}:  $\Gamma/N$ is the subdivision of a connected simple $4$-valent graph $\Lambda$ of order divisible by $3$.
\medskip

By Lemma~\ref{lemma:technical}, $G/K$ acts faithfully on $\Lambda$. Finally, since $\Gamma/N$ is locally $G/K$-arc-transitive, $\Lambda$ is $G/K$-arc-transitive.

\medskip

\noindent
{\bf Fact 9}:  $G/K$ acts faithfully on $\Lambda$ and  $\Lambda$ is $G/K$-arc-transitive.
\medskip

Since $3$ divides $|\V(\Lambda)|$ and $G/K$ is solvable, it follows from Lemma~\ref{lemma:main} that there exists an element $Kg\in G/K$ of order $3$ acting semiregularly on $\V(\Lambda)$. Since $3$ is odd, $Kg$ also acts semiregularly on $\E(\Lambda)$ and hence on $\V(\Gamma/N)$. Since $3$ is coprime to $|K|$, Lemma~\ref{lemma:coprime} implies that $G$ contains a semiregular element of order $3$. This concludes the proof.

\subsection{Proof of Corollary~\ref{cor:poly}}
Let $\Gamma$ be a connected, edge-transitive graph of valency at most $4$. Suppose first that $\Gamma\cong\W(2^n,2)$ for some $n\geq 2$. Let $g$ be the automorphism of $\W(n,2)$ that swaps each pair $\{(i,0), (i,1)\}$. Note that $g$ is semiregular and, since it swaps vertices that are distance $2$ apart, it preserves the parts of the bipartition of $\Gamma$.  This concludes the proof in this case.

Suppose now that $\Gamma\cong \DD\Lambda$ for some arc-transitive, $4$-valent graph $\Lambda$ of order a power of $2$. By Corollary~\ref{cor:v+e}, there exists $g\in\Aut(\Lambda)$ such that $g$ acts semiregularly on each of $\V(\Lambda)$ and $\E(\Lambda)$. This induces an automorphism of $\Gamma=\DD\Lambda$ which maps a vertex $(v,i)$ to the vertex $(v^g,i)$ and a vertex $e\in \E(\Lambda)$ to the vertex $e^g$. This automorphism of $\Gamma$ is clearly semiregular on $\V(\Gamma)$ and preserves the two parts of the bipartition of $\Gamma$.

We now assume that we are not in one of the two cases above. By Theorems~\ref{theorem:main3} and~\ref{theorem:main}, $\Aut(\Gamma)$ contains a semiregular element. If $\Gamma$ is bipartite and locally arc-transitive, let $G$ be the subgroup of $\Aut(\Gamma)$ preserving the two parts of the bipartition. By Lemma~\ref{lem:final}, $\Gamma$ is $G$-edge-transitive and applying  Theorems~\ref{theorem:main3} and~\ref{theorem:main} again, we see that $G$ contains a semiregular element. This completes the proof of Corollary~\ref{cor:poly}.

\section{Examples}\label{sec:counterexample}
\begin{lemma}
Let $n\geq 3$, let $\Gamma=W(2^n,2)$ and recall that $\V(\Gamma)=\ZZ_{2^n}\times\ZZ_2$. Let $\sigma$ and $\tau$ be permutations of $\ZZ_{2^n}$ according to the rules
\begin{align*}
\sigma&=(1~3~5\ldots~2^n-1)(0~2~4\ldots~2^n-2),\\
\tau&\colon i\mapsto -i+4.
\end{align*}

The permutations $\sigma$ and $\tau$ have an induced action on $\V(\Gamma)$ by acting on the first coordinate. They are clearly automorphisms of $\Gamma$. For $i\in\{1,2,\ldots,2^n\}$, let $\alpha_i$ be the automorphism of $\Gamma$ interchanging $(i,0)$ and $(i,1)$ and fixing all other vertices of $\Gamma$. Let $x= \alpha_1\sigma$ and let $z=\alpha_0\alpha_2\alpha_3\ldots\alpha_{2^n-4}\alpha_{2^n-2}\alpha_{2^n-1}$, that is $z$ is the product of all $\alpha_i$'s with $i$ not congruent to $1\pmod 4$. Let $G=\langle x,\tau,z\rangle$. Then $G$ is an edge-regular group of automorphisms of $\Gamma$ of order $2^{n+2}$ and every involution of $G$ fixes a vertex.
\end{lemma}
\begin{proof}
Clearly $\langle\sigma,\tau\rangle\cong D_{2^{n-1}}$ and $|\tau|=|z|=2$. Note that $|\sigma|=2^{n-1}$ and $x^{2^{n-1}}=\alpha_1\alpha_3\alpha_5\ldots\alpha_{2^n-1}$ hence $|x|=2^n$. Moreover,
$$\tau x\tau x=\tau \alpha_1\sigma\tau \alpha_1\sigma=\tau \alpha_1\tau\sigma^{-1} \alpha_1\sigma=\alpha_1^\tau\alpha_1^{\sigma}=\alpha_3\alpha_3=1,$$
hence $\langle x,\tau\rangle \cong D_{2^n}$.

Note that $z^\tau=\alpha_0\alpha_1\alpha_2\ldots\alpha_{2^n-4}\alpha_{2^n-3}\alpha_{2^n-2}$, that is $z^\sigma$ is the product of all $\alpha_i$'s with $i$ not congruent to $3\pmod 4$. It follows that $\langle z,z^\tau\rangle\cong C_2^2$. Note that $z^\sigma=z^\tau$. It follows that $\langle z,z^\tau\rangle$ is normal in $G$. Note that $zz^\tau=x^{2^{n-1}}$ and hence $\langle x,\tau\rangle\cap\langle z,z^\tau\rangle=\langle x^{2^{n-1}}\rangle$ and hence $|G|=\frac{|\langle x,\tau\rangle||\langle z,z^\tau\rangle|}{|\langle x^{2^{n-1}}\rangle|}=2^{n+2}$.

We now determine representatives for the conjugacy classes of involutions of $G$. Since $\langle z,z^\tau\rangle$ is normal in $G$ and $\langle x,\tau\rangle \cong D_{2^{n}}$, it follows that every involution in $G$ is conjugate to an element of one of the following three cosets: $\langle z,z^\tau\rangle$, $\langle z,z^\tau\rangle \tau$ and $\langle z,z^\tau\rangle x\tau$. 

The involutions in $\langle z,z^\tau\rangle$ are $z$, $z^\tau$ and $x^{2^{n-1}}$, but $z$ is conjugate to $z^\tau$. The only involutions in $\langle z,z^\tau\rangle \tau$ are $\tau$ and  $x^{2^{n-1}}\tau$, but $\tau$ is conjugate to $x^{2^{n-1}}\tau$ under $x^{2^{n-2}}$. The four elements of  $\langle z,z^\tau\rangle x\tau$ are involutions, but $x\tau$ and $x^{2^{n-1}}x\tau$ are conjugate under $x^{2^{n-2}}$. We have shown that every involution in $G$ is conjugate to one of $z$, $x^{2^{n-1}}$, $\tau$, $x\tau$, $zx\tau$ or $z^\tau x\tau$. 

We show that each of them fixes a vertex of $\Gamma$. Clearly, $z$ fixes the vertex $(1,0)$ while $x^{2^{n-1}}$ and $\tau$ both fix the vertex $(2,0)$. It is not too hard to check that $x\tau$ and $z$ both fix $(2^{n-1}+1,0)$ hence so does $zx\tau$. Finally, $z^\tau x\tau$ fixes the vertex $(1,0)$.

We now show that $\Gamma$ is $G$-edge-transitive. There is a natural $G$-invariant partition of $\V(\Gamma)$ induced by the block $B=\{(1,0),(1,1)\}$. Clearly $\langle\sigma\rangle$ has two orbits on these blocks and $x$ maps $(1,0)$ to $(3,1)$. It follows that $G$ has at most two orbits on vertices. In fact, it is easy to see that $G$ has exactly two orbits on vertices. Let $u$ be the vertex $(2,0)$ and $v$ be the vertex $(2^{n-1}+1,0)$. Note that $u$ and $v$ are in different $G$-orbits. Moreover, $G_u$ contains $\tau$ and $x^{2^{n-1}}$, and hence $G_u^{\Gamma(u)}$ is transitive. Similarly, $G_v$ contains $x\tau$ and $z$ and hence $G_v^{\Gamma(v)}$ is transitive. Since $G$ has two orbits on vertices, it follows that $\Gamma$ is locally $G$-arc-transitive and hence $G$-edge-transitive by Lemma~\ref{lem:final}. Since $\Gamma$ has $2^{n+2}$ edges, $G$ is actually edge-regular.
\end{proof}

\begin{lemma}
Let $n\geq 2$, let $\Gamma=W(2,2^n)$ and label the vertices of $\Gamma$ by $\{1,1',2,2',\ldots, 2^n, (2^n)'\}$ such that $i$ and $i'$ both have  $\{i-1,(i-1)',i+1,(i+1)'\}$ as their neighbourhood (addition is computed modulo $2^n$). Let $t$, $\sigma$ and $\tau$ be permutations of $\V(\Gamma)$ according to the rules
\begin{align*}
t&=(1~1')(3~3')(5~5')\ldots (2^{n}-1~(2^{n}-1)'),\\
\sigma&=(1~2~3~\ldots~2^n)(1'~2'~3'~\ldots~(2^n)'),\\
\tau&\colon i\mapsto -i+3, \quad\tau\colon i'\mapsto (-i+3)',
\end{align*}
again, with addition computed modulo $2^n$. Clearly, $t, \sigma,\tau\in\Aut(\Gamma)$. 

Let $\Sigma=\mathbb{D}(\Gamma)$ and recall that $\V(\Sigma)=\E(\Gamma)\cup(\V(\Gamma)\times\ZZ_2)$ and $e\in\E(\Gamma)$ is adjacent to $(v,i)$ in $\Sigma$ if $v$ is incident with $e$ in $\Gamma$. For $v\in\V(\Gamma)$, let $\alpha_v$ be the automorphism of $\Sigma$ interchanging $(v,0)$ and $(v,1)$ and fixing all other vertices of $\Sigma$. Further, there is a natural embedding of $\Aut(\Gamma)$ in $\Aut(\Sigma)$ by letting $g\in\Aut(\Gamma)$ mapping a vertex $e$ in $\E(\Gamma)$ to $e^g$ and a vertex $(v,i)$ in $\V(\Gamma)\times\ZZ_2$ to $(v^g,i)$. Note that the group generated by $\Aut(\Gamma)$ and $\{\alpha_v\mid v\in\V(\Gamma)\}$ is isomorphic to the wreath product $\ZZ_2 \wr \Aut(\Gamma)$.

Let $\omega=\alpha_1\alpha_3\ldots\alpha_{2^n-1}$, let $x=\alpha_1\alpha_{1'}\sigma$ and let $z=\omega t$. Let $G=\langle x,\tau,z\rangle$. Then $G$ is an edge-transitive group of automorphisms of $\Sigma$ of order $2^{n+5}$ such that every involution fixes a vertex.
\end{lemma}
\begin{proof}

Note that $|\tau|=2$, $|\sigma|=2^n$ and $x^{2^n}=\alpha_1\alpha_{1'}\alpha_2\alpha_{2'}\ldots\alpha_{2^n}\alpha_{(2^n)'}$ hence $|x|=2^{n+1}$. Note that $\sigma\tau$ is an involution fixing both $1$ and $1'$ hence $\sigma\tau$ commutes with both $\alpha_1$ and $\alpha_{1'}$ and $(x\tau)^2=1$. It follows that $\langle x,\tau\rangle \cong D_{2^{n+1}}$. Observe that $z^2=\alpha_1\alpha_{1'}\alpha_3\alpha_{3'}\ldots\alpha_{2^n-1}\alpha_{(2^n-1)'}$ hence $|z|=4$. Note also that $(z^\tau)^2=\alpha_2\alpha_{2'}\alpha_4\alpha_{4'}\ldots\alpha_{2^n}\alpha_{(2^n)'}$ hence $x^{2^n}=z^2(z^\tau)^2$.

Note that $t$ commutes with both $t^\tau$ and $\omega^\tau$ Since $\tau$ is an involution, it follows that $t^\tau$ commutes with $\omega$. Thus
$$zz^\tau=\omega t\omega^\tau t^\tau=\omega \omega^\tau t^\tau t=\omega^\tau t^\tau \omega t=z^\tau z.$$

This implies that $\langle z,z^\tau\rangle \cong C_4^2$.  Observe that $\alpha_1\alpha_{1'}$ commutes with both $t$ and $t^\tau$ and hence $\alpha_1\alpha_{1'}$ commutes with both $z$ and $z^\tau$.

In particular,  $z^x=z^\sigma=\omega^\sigma t^\sigma=\omega^\tau t^\tau=z^\tau$ and ${z^\tau}^x={z^\tau}^\sigma=z^{\sigma\tau^{-1}}=z^{\tau\tau^{-1}}=z$. It follows that $\langle z,z^\tau\rangle$ is normal in $G$. Clearly $\langle x,\tau\rangle\cap\langle z,z^\tau\rangle=\langle x^{2^{n}}\rangle$ and hence $|G|=\frac{|\langle x,\tau\rangle||\langle z,z^\tau\rangle|}{|\langle x^{2^{n}}\rangle|}=2^{n+5}$.

We now determine representatives for the conjugacy classes of involutions of $G$. Since $\langle z,z^\tau\rangle$ is normal in $G$ and $\langle x,\tau\rangle \cong D_{2^{n+1}}$, it follows that every involution in $G$ is conjugate to an element of one of the following three cosets: $\langle z,z^\tau\rangle$, $\langle z,z^\tau\rangle \tau$ and $\langle z,z^\tau\rangle x\tau$.

Since $\langle z,z^\tau\rangle\cong C_4^2$, the only involutions in $\langle z,z^\tau\rangle$ are $z^2$, $({z^\tau})^2$ and $x^{2^n}$, but $z^2$ is conjugate to $({z^\tau})^2$ under $\tau$. The only involutions in $\langle z,z^\tau\rangle \tau$ are $\tau$, $z{(z^\tau)}^3\tau$, $z^3z^\tau\tau$ and   $x^{2^n}\tau$, but $\tau$ is conjugate to $x^{2^n}\tau$ under $x^{2^{n-1}}$, to $z{(z^\tau)}^3\tau$ under $z^\tau$ and to $z^3z^\tau\tau$ under $z$. Similarly, the only involutions in $\langle z,z^\tau\rangle x\tau$ are $x\tau$, $z^2x\tau$, ${(z^\tau)}^2x\tau$ and $x^{2^n}x\tau$, but $x\tau$ and $x^{2^n}x\tau$ are conjugate under $x^{2^{n-1}}$. We have shown that every involution in $G$ is conjugate to one of $z^2$, $x^{2^n}$, $\tau$, $x\tau$, $z^2x\tau$, or ${(z^\tau)}^2x\tau$. 

We show that each of them fixes a vertex of $\Sigma$. By definition, $z^2$ and $x^{2^n}$ fix all the vertices of $\Sigma$ corresponding to $\E(\Gamma)$. Similarly, $\tau$ fixes the edge $\{1,2\}$ of $\Gamma$ and hence fixes the corresponding vertex of $\Sigma$. It is not hard to check that $x\tau$ fixes the vertex $((2^{n-1}+1),0))$ of $\Sigma$. Since ${(z^\tau)}^2$ also fixes $((2^{n-1}+1),0)$, so does ${(z^\tau)}^2x\tau$. Finally, $z^2x\tau$ fixes $(1,0)$.

We now show that $G$ is transitive on edges of $\Sigma$. Clearly, $\langle \sigma,t\rangle$ acts transitively on $\V(\Gamma)$. Moreover, the group $\langle t^\tau,\sigma\tau\rangle$ fixes the vertex $1$ and acts transitively on the set $\{2,2',2^{n},(2^{n})'\}$ of neighbours of $1$ in $\Gamma$. Hence $\Gamma$ is $\langle \sigma,t,\tau\rangle$-arc-transitive. This implies that $G$ is transitive on $\E(\Gamma)$. Let $e=\{1,2\}$ viewed as a vertex of $\Sigma$. The group $\langle \tau,x^{2^n}\rangle$ fixes $e$ and acts transitively on the set $\{(1,0),(1,1),(2,0),(2,1)\}$ of neighbours of $e$ in $\Sigma$. Since every edge of $\Sigma$ has one endpoint in $\E(\Gamma)$, $G$ is transitive on $\E(\Sigma)$.
\end{proof}

%%%%%%%%%%%%%%%%%%%%%%%%%%%%%%%%%%%%%%%%%%%%%%%%%%%%%%%%%%%%%%%%%%%%%%%%%%%%%%%
\bibliographystyle{amsplain}

\end{document}